\documentclass[12pt,a4paper]{amsart}

\usepackage[utf8]{inputenc}

\usepackage{accents}

\usepackage{mathrsfs}
\usepackage{todonotes}
\usepackage[utf8]{inputenc}

\usepackage{amssymb, tikz}

\usepackage{hyperref}

\usepackage{amsmath}

\newtheorem{theorem}{Theorem}[section]

\newtheorem{lemma}[theorem]{Lemma}

\theoremstyle{definition}
\newtheorem{definition}[theorem]{Definition}

\theoremstyle{remark}
\newtheorem{remark}[theorem]{Remark}

\DeclareMathOperator{\GCH}{GCH}
\DeclareMathOperator{\PFA}{PFA}

\DeclareMathOperator{\WCG}{WCG}

\newcommand{\CH}{{\rm CH}}

\newcommand{\Lim}{{\rm Lim}}

\newcommand{\dom}{{\rm dom}}

\newcommand{\bbR}{{\mathbb R}}

\newcommand{\cH}{{\mathscr H}}

\newcommand{\bbP}{{\mathbb P}}

\newcommand{\bbQ}{{\mathbb Q}}

\newcount\skewfactor
\def\mathunderaccent#1#2 {\let\theaccent#1\skewfactor#2
\mathpalette\putaccentunder}
\def\putaccentunder#1#2{\oalign{$#1#2$\crcr\hidewidth
\vbox to.2ex{\hbox{$#1\skew\skewfactor\theaccent{}$}\vss}\hidewidth}}
\def\name{\mathunderaccent\tilde-3 }

\begin{document}

\title[A generic absoluteness principle consistent with  large continuum]{A generic absoluteness principle consistent with large continuum}

\author[M. Golshani]{Mohammad Golshani}
\address{School of Mathematics\\
 Institute for Research in Fundamental Sciences (IPM)\\
  P.O. Box:
19395-5746\\
 Tehran-Iran.}
\email{golshani.m@gmail.com}
\urladdr{http://math.ipm.ac.ir/~golshani/}

\thanks{The author's research has been supported by a grant from IPM (No. 1402030417). The results of this paper are motivated from discussions with
David Asper\'{o}, to whom the author is very grateful. In particular, the current presentation of Theorem \ref{thm2} is suggested by him.
He also thanks David Asper\'{o} and Rahman Mohammadpour for their useful comments.}



\subjclass[2010]{Primary: 03E50, 03E35, 03E65}

\keywords { large continuum, generic absoluteness, memory iteration, $\aleph_2$-p.i.c.}

\begin{abstract}
We state a new generic absoluteness principle, and use Shelah's memory iteration technique to show that it is consistent with the large continuum.
\end{abstract}

\maketitle
\setcounter{section}{-1}


\section{Introduction}
Generic absoluteness principles are widely studied in set theory. They assert that certain properties of the
set-theoretic universe cannot be changed by forcing.
For instance Shoenfield in \cite{shoenfield}, showed that if a sentence $\psi$ is $\Sigma^1_2$ or $\Pi^1_2$, it is impossible to reach
a forcing extension in which $\psi$ holds if it does not already hold in the ground model.
See   \cite{aspero}, \cite{bagaria1}, \cite{bagaria2},
\cite{bagaria4}, \cite{bagaria5}, \cite{bagaria3}, \cite{friedman}, \cite{To1}, \cite{viale}, \cite{wilson}
and \cite{woodin1} for more on the subject.

In this paper we introduce another generic absoluteness principle and show its consistency with large values of the continuum.
Our main result is the following:
\begin{theorem}\label{thm2} Assume $\GCH$. Let $\kappa\geq\aleph_2$ be a regular cardinal. Then there is a proper partial order $\bbP$ with the $\aleph_2$-chain condition and forcing the following statements.
\begin{enumerate}
\item[$(\dag)$] $2^{\aleph_0} = \kappa$,

\item[$(\ddag)$] Suppose $\varphi(x, y)$ is a restricted formula in the language of set theory. Suppose for every $a\in \cH(\aleph_2)$ and every inner model  $M$ of the form $M=\bold V[\bold H]$, where $\bold H$ is $\bbR$-generic over $\bold V$, for some $\bbR \lessdot \bbP$,
    if $x\in M$, $\omega_1^M=\omega_1$, and $M\models\CH$, then there is an $\bbR$-name $\name{\bbQ}$ such that $\name{\bbQ}$ is $(\bbR, \bbP)$-proper   \footnote{See Section \ref{perl} for the definition of $(\bbR, \bbP)$-properness.} with the $\aleph_2$-p.i.c.\ and forcing $\cH(\aleph_2)\models (\exists y)\varphi(a, y)$. Then $\cH(\aleph_2)\models(\forall x)(\exists y)\varphi(x, y)$.
\end{enumerate}
\end{theorem}
\begin{remark}
A generalization of the above theorem has recently announced in \cite{aspero-gol}. However our proof is here completely different, and might be of some interest. In particular, we develop a general theory of countable support memory iteration, and prove some general results about it.
\end{remark}
The proof uses a countable support version of Shelah's memory iteration. See  \cite{golshani}, \cite{sh1102}, \cite{sh684},  \cite{sh592} and \cite{sh619}
 for more on the method.
The paper is organized as follows. In Section \ref{perl}, we recall some definitions and results which are used in this paper, and then in Section
\ref{memory}, we develop a general case of countable support memory iteration and prove some results about it. Then in Section
\ref{sectionproof}
we present the proof of Theorem \ref{thm2}.

\section{Some preliminaries}
\label{perl}

In this section we present some definitions and results that we will make use in the next sections.
\subsection{$(\bbP_0, \bbP_1)$-properness}
In this subsection we define the following notion of properness. It will help us to show that some memory iterations of proper forcing notions are
again proper.
\begin{definition}
Suppose  $\bbP_0 \lessdot \bbP_1$ are forcing notions,  and $\name{\bbQ}\in \bold V$ is a $\bbP_0$-names  for a  forcing notion.
We  say that $\name{\bbQ}$ is \emph{$(\bbP_0, \bbP_1)$-proper} if for every large enough regular cardinal $\chi$ and every countable $M\prec \cH(\chi)^{\bold V}$ such that  $\bbP_0$, $\bbP_1, \name{\bbQ}\in M$,
if $(p, \name{q}) \in (\bbP_1 \ast \name{\bbQ}) \cap M$, then there is $\name{q}^*$ such that:
\begin{enumerate}
\item $\name{q}^*$ is a $\bbP_0$-name for an element of $\name{\bbQ}$,

\item If $\bold G$ is $\bbP_1$-generic over $\bold V$ and $\bold V \cap  M[\bold G]=M$ and $p \in \bold G$,
then
\begin{center}
$\bold V[\bold G] \models$``$\name{q}^*_{(\bold G  \cap \bbP_0)}$ \footnote{Here by $\bold G  \cap \bbP_0$ we mean the canonical $\bbP_0$-generic filter obtained from $\bold G$, using the hypothesis $\bbP_0 \lessdot \bbP_1$.} is an $(M[\bold G], \name{\bbQ}_{\bold G})$-generic condition, extending $\name{q}_{\bold G}$''.
\end{center}
\end{enumerate}
\end{definition}

\begin{definition}
Suppose $\name{\bbQ}\in \bold V$ is a $\bbP_0$-names  for a  forcing notion. We say  $\name{\bbQ}$ is \emph{$\bbP_0$-stably  proper} if for each proper forcing notion $\bbP_1$ such that $\bbP_0 \lessdot \bbP_1$
  and $\bbP_1 / \dot{\bold{G}}_{\bbP_0}$ is proper, $\name{\bbQ}$ is  $(\bbP_0, \bbP_1)$-proper.
\end{definition}

\subsection{The $\aleph_2$-properness isomorphism condition}
The following notion is due to Shelah \cite[VIII, Section 2]{Sh:f}.
\begin{definition}  Given a partial order $\bbQ$, we say that \emph{$\bbQ$ has  the $\aleph_2$-properness isomorphism condition ($\aleph_2$-p.i.c.)} in case for every large enough regular cardinal $\theta$ and for any two countable $N_0$, $N_1\prec H(\theta)$ such that $\omega_2$, $\bbQ\in N_0\cap N_1$,  if $\pi:(N_0; \in)\to (N_1; \in)$ is  an isomorphism, $N_0\cap N_1\cap\omega_2$ is a proper initial segment of both $N_0\cap\omega_2$ and $N_1\cap\omega_2$, and $N_0\cap\omega_2\subseteq \min((N_1\cap\omega_2)\setminus N_0)$, then for every $p\in\bbQ\cap N_0$ there is a condition $q\in\bbQ$ extending $p$ and such that
\begin{enumerate}
\item $q\Vdash_{\bbQ}\mbox{``}(\forall r\in N_0\cap\bbQ)(r\in\name G_{\bbQ}\mbox{ iff }\pi(r)\in \name G_{\bbQ})\mbox{''}$,
\item $q\Vdash_{\bbQ} p\in \name G_{\bbQ}$, and
\item $q$ is $(N_0, \bbQ)$-generic.
\end{enumerate}
\end{definition}

It is a standard fact that if $\CH$ holds, then every partial order with the $\aleph_2$-p.i.c.\ has the $\aleph_2$-c.c.  Also, the following is standard and well-known, see \cite[Ch. VIII]{Sh:f}, or the proof of \cite[Theorems 2.10 and 2.12]{abraham}).

\begin{lemma}\label{a2-pic-cc} Suppose $\bbP=\langle \langle \bbP_\alpha: \alpha \leq \delta \rangle, \langle \name{\bbQ}_\beta: \beta < \delta \rangle\rangle$ is a countable support iteration and for each $\beta<\delta$, $\Vdash_{\bbP_\beta}\mbox{`` }\name\bbQ_\beta\mbox{ is proper and has the $\aleph_2$-p.i.c.''}$. Then:

\begin{enumerate}
\item if $\delta<\omega_2$, then $\bbP_\delta$ has the $\aleph_2$-p.i.c.;
\item if $\delta\leq\omega_2$ and $\CH$ holds, then $\bbP_\delta$ has the $\aleph_2$-c.c.;
\item if $\delta < \omega_2$ and $\CH$ holds, then $\CH$ holds in $V^{\bbP_\delta}.$
\end{enumerate}
\end{lemma}

\section{Shelah's memory iteration}
\label{memory}
In this section we give a general version of  Shelah's memory iteration with countable supports,
and study some of its properties.
Suppose  $\kappa \geq \aleph_2$ is an uncountable regular cardinal. A countable support memory iteration of length $\kappa$ is defined as a sequence
\[
\bold{\mathfrak{q}}= \langle\bbP_\alpha, \name{\bbQ}_\beta, \mathcal{U}_{\beta}: \alpha \leq \kappa, \beta < \alpha \rangle
\]
where:
\begin{enumerate}


\item $\mathcal{U}_{\beta} \in [\beta]^{<\kappa}$;

\item if $\xi \in \mathcal{U}_{\beta}$, then $\mathcal{U}_{\xi}   \subseteq \mathcal{U}_{\beta}$;

\item $\bbP_{\mathcal{U}_\beta}=\{p \in \bbP: \dom(p) \subseteq \mathcal U_\beta      \} \lessdot \bbP_\beta$;

\item $\name{\bbQ}_\beta$ is a $\bbP_{\mathcal U_\beta}$-name for a forcing notion;

    \item for each $\alpha \leq \kappa$ we have
     $p \in \bbP_\alpha$ iff
        \begin{enumerate}
        \item $p$ is a function with $\dom(p)$ a countable subset of $\alpha$;

        \item for each $\beta  \in \dom(p),  p(\beta)$ is a $\bbP_{\mathcal U_\beta}$-name and  $p \restriction \mathcal U_\beta \Vdash_{\bbP_{\mathcal U_\beta}}$``$p(\beta) \in \name{\bbQ}_\beta$'';
        \end{enumerate}
\end{enumerate}

Given $\bbP_\alpha$-conditions $p_0$, $p_1$, we define $p_1\leq_{\bbP_\alpha} p_0$ ($p_1$ extends $p_0$) iff
\begin{enumerate}
\item $\dom(p_0)\subseteq \dom(p_1)$ and
\item for every $\beta\in \dom(p_0)$, $p_0\restriction\beta\Vdash_{\bbP_\beta}\mbox{``}p_1(\beta)\leq_{\name{\bbQ}_\beta}p_0(\beta)\mbox{''}$. \end{enumerate}
We now study the properness condition for such iterations. The following theorem gives a general criteria
for a memory iteration as above to be proper.
\begin{theorem}
\label{properness} Suppose $\bold{\mathfrak{q}}$ is as above, and suppose that for each $\beta<\kappa$,
$\name{\bbQ}_\beta$
 is $(\bbP_{\mathcal U_{\beta}}, \bbP_\beta)$-proper. Then, for each $\alpha \leq \kappa$:
 \begin{enumerate}
 \item[(a)] $\bbP_\alpha$ is proper,
  \item[(b)] $\bbP_\alpha / \dot{\bold{G}}_{\bbP_{\mathcal U_\alpha}}$ is proper.
 \end{enumerate}
\end{theorem}
\begin{proof}
By induction on $\alpha \leq \kappa$ we show that:
\begin{enumerate}
\item[$(\ast)_\alpha$] if $\beta <\alpha$, $M \prec \cH(\chi)$ is countable, where $\chi > \kappa$ is large enough regular and $\bold{\mathfrak{q}}, \alpha, \beta \in M,$  if $p \in M \cap \bbP_\alpha$ and if $p' \leq p \restriction \beta$ is $(\bbP_\beta, M)$-generic, then there is $p'' \in \bbP_\alpha$
    such that $p''\leq p$ and $p''$ is $(\bbP_\alpha, M)$-generic.
\end{enumerate}
 The successor case $\alpha=\beta+1$ follows from the assumption
$\name{\bbQ}_\beta$
 is $(\bbP_{\mathcal U_{\beta}}, \bbP_\beta)$-proper, and the limit case can be treated as usual proof of properness
 in limit cases.
\end{proof}
In the next theorem we give general conditions under which the iteration has good chain conditions. We first start by fixing some definitions and notations.
\begin{definition}
Suppose  $\bold{\mathfrak{q}}$ is as above.
\begin{enumerate}
\item We say $A \subseteq \kappa$ is $\bold{\mathfrak{q}}$-closed, if
\[
\beta \in A \Rightarrow \mathcal U_\beta \subseteq A.
\]

\item if $\alpha \leq \kappa$ and   $A \subseteq \alpha$ is $\bold{\mathfrak{q}}$-closed, then
\[
\bbP^\alpha_A= \{p \in \bbP_\alpha: \dom(p) \subseteq A  \}.
\]
We also set $\bbP_A=\bbP^\kappa_A.$
\end{enumerate}
It is easily seen that if $A \subseteq \alpha$ is $\bold{\mathfrak{q}}$-closed, then $\bbP^\alpha_A \lessdot \bbP_\alpha.$
\end{definition}
We also have an analogue of Lemma \ref{a2-pic-cc}, whose proof is essentially the same.
\begin{lemma}
\label{lem1} Suppose $\bold{\mathfrak{q}}$ is as above,  $A \subseteq \kappa$ is $\bold{\mathfrak{q}}$-closed
of size $\leq \aleph_1$ and for each $\beta \in A$, $\Vdash_{\bbP_{\mathcal U_\beta}}$`` $\name{\bbQ}_\beta$ satisfies the $\aleph_2$-p.i.c.''.
Then:
\begin{enumerate}
\item $\bbP_A$ is $\aleph_2$-c.c.,

\item If $\CH$ holds,  then it holds in the generic extension by $\bbP_A$.
\end{enumerate}
\end{lemma}

   \begin{theorem}
   \label{lem3}
Suppose  $\bold{\mathfrak{q}}$ is as above, and suppose that
 \begin{enumerate}
\item for each $\beta<\kappa$, $\mathcal U_\beta \in [\beta]^{\leq \aleph_1},$

\item for each $\beta<\kappa$, $\Vdash_{\bbP_{\mathcal U_\beta}}$`` $\name{\bbQ}_\beta$ satisfies the $\aleph_2$-p.i.c.''.
\end{enumerate}
Then $\bbP_\kappa$ is $\aleph_2$-c.c.
\end{theorem}
\begin{proof}
Let $\bar p= \langle  p_\xi: \xi < \omega_2     \rangle$
be a set of conditions in $\bbP_{\kappa}$. For each $\xi < \omega_2$
pick a $\bold{\mathfrak{q}}$-closed set $A_\xi$ of size $\aleph_1$ such that
$$\dom(p_\xi) \subseteq A_\xi.$$
We may suppose that $\langle A_\xi: \xi < \omega_2    \rangle$
is increasing.
By shrinking the sequence $\bar p$ if necessary, suppose that
$\{\dom(p_\xi) : \xi < \omega_2\}$ forms a $\Delta$-system with root $d$. Let $i_* < \omega_2$ be such that $d \subseteq A_{i_*}$ . As $|A_{i_*}|\leq \aleph_1$, by Lemma \ref{lem1}(1)
 $\bbP_{A_{i_*}}$ satisfies the $\aleph_2$-c.c., so for some $\xi < \zeta < \omega_2$, $p_\xi \restriction A_{i_*}$ is compatible with
$p_\zeta \restriction A_{i_*}$. But then it is easily seen that $p_\xi$ and $p_\zeta$ are compatible.
\end{proof}

\section{Proof of Theorem \ref{thm2}}
\label{sectionproof}

In this section we prove Theorem \ref{thm2}. Thus suppose $\GCH$ holds and let $\kappa \geq \aleph_2$ be a regular cardinal. Let
$\phi: \kappa \to \cH(\kappa)$
be such that for each $x \in \cH(\kappa)$, the set $\phi^{-1}(x) \subseteq \kappa$
is unbounded.
Let
\[
\bold{\mathfrak{q}}= \langle\bbP_\alpha, \name{\bbQ}_\beta, \mathcal{U}_{\beta}: \alpha \leq \kappa, \beta < \alpha \rangle
\]
be a memory iteration as in Section \ref{memory}, with the following additional properties:
\begin{enumerate}
\item[(6)] for each $\beta < \kappa,$ $\name{\bbQ}_\beta$ is $(\bbP_{\mathcal{U}_\beta}, \bbP_\beta)$ proper and is forced by  $\bbP_{\mathcal{U}_\beta}$  to be  $\aleph_2$-p.i.c.

\item[(7)] Assume $\beta < \kappa.$
\begin{enumerate}
    \item if $\phi(\beta)=(A, \varphi(x, y), \name{x})$, where $A \in [\beta]^{\leq \aleph_1}$ is $(\bold{\mathfrak{q}}\restriction \beta)$-closed,  $\varphi(x, y)$ is a restricted formula in the language of set theory
    and $\name x\in \cH(\kappa)$ is a $\bbP^\beta_{A}$-name for a member of $\cH(\aleph_2)$,
    and if there is  a $\bbP^\beta_{A}$-name  $\name{\bbQ}$ for a  $(\bbP^\beta_{A}, \bbP_\beta)$-proper poset with the $\aleph_2$-p.i.c.\ forcing $\cH(\aleph_2)\models (\exists y)\varphi(\name x, y)$, then $\name{\bbQ}_\beta$ is a name for such a poset. We also set
    $\mathcal U_\beta=A.$


    \item Otherwise we let $\name{\bbQ}_\beta$  be the trivial forcing notion, and set $\mathcal U_\beta=\emptyset.$
    \end{enumerate}
 \end{enumerate}
Let us show that the forcing notion $\bbP=\bbP_\kappa$ is as required.
We can easily show that $2^{\aleph_0}=\kappa$ holds in $V^{\bbP_\kappa}$, thus item $(\dag)$
of Theorem \ref{thm2} is satisfied. Let us show that item $(\ddag)$
holds in $V^{\bbP_\kappa}$.

Thus assume $\varphi(x, y)$ is a restricted formula in the language of set theory, and suppose for every $a\in \cH(\aleph_2)$ and every inner model $M$ of the form $M=\bold V[\bold H]$, where $\bold H$ is $\bbR$-generic over $\bold V$, for some $\bbR \lessdot \bbP$, if $x\in M$, $\omega_1^M=\omega_1$, and $M\models\CH$, then there is an $\bbR$-name $\name{\bbQ}$ such that $\name{\bbQ}$ is $(\bbR, \bbP)$-proper  with the $\aleph_2$-p.i.c.\ and forcing $\cH(\aleph_2)\models (\exists y)\varphi(a, y)$.
Let also $\name x$ be a $\bbP_\kappa$-name for an element of $\cH(\aleph_2)$. Let $\beta < \kappa$ be such that:
\begin{itemize}
\item $\phi(\beta)=(A, \varphi(x, y), \name{x})$,

\item $A \in [\beta]^{\leq \aleph_1}$ is $\bold{\mathfrak{q}}$-closed,

\item $\name x$ is a $\bbP_{A}$-name,

\item there is a $\bbP^\beta_A$-name $\name{\bbQ}$ such that $\name{\bbQ}$ is $(\bbP^\beta_A, \bbP_\beta)$-proper  with the $\aleph_2$-p.i.c.\ and forcing $\cH(\aleph_2)\models (\exists y)\varphi(a, y)$.
\end{itemize}
By Lemmas \ref{lem1} and \ref{lem3} and our assumption, such a $\beta$ exists.
It then follows that $\name{\bbQ}_\beta$ is such a forcing notion. From this,  we can easily conclude the result.
The theorem follows.

\end{document}